\documentclass[12pt,reqno]{amsart}

\usepackage{epsfig}
\usepackage{amsmath,amsthm}
\usepackage{amsfonts}
\usepackage{latexsym}
\usepackage{amsbsy}
\usepackage{amsfonts,amssymb,amscd,amsmath,enumerate,verbatim,calc,eucal,enumerate}

\numberwithin{equation}{section}

\usepackage{amsfonts,amssymb,amscd,amsmath,enumerate,verbatim,calc,eucal,enumerate}

\newcommand{\m}{\CMcal}

\theoremstyle{plain}
\newtheorem{theorem}{Theorem}[section]
\newtheorem{corollary}[theorem]{Corollary}

\newtheorem{proposition}[theorem]{Proposition}

\newtheorem{definition}[theorem]{Definition}

\begin{document}
	
	\title[]{A simple proof of the fixed point theorem in $C^*$-algebra valued metric space}
	\date{\today}

	\author[Trivedi]{Harsh Trivedi}
	\address{ 	Silver Oak College of Engineering and Technology,
		Near Bhagwat Vidyapeeth, Ahmedabad-380061, India}
	\email{trivediharsh26@gmail.com}

	\begin{abstract}
		
		We obtain a fundamental inequality for a contraction with respect to a $C^*$-algebra valued metric space. As an application of this inequality a simple direct proof is given for the fixed point theorem in $C^*$-algebra valued metric space.\\\\
		

	\end{abstract}

	\maketitle
	
	\section{Introduction}
	
	In \cite{ZLH14}, Ma, Jiang, and Sun defined $C^*$-algebra valued metric spaces and the corresponding contractions. We first recall both these definitions:
	\begin{definition}
		Let $X$ be a non empty set and let $\m A$ be a unital $C^*$-algebra. Then the function $d:X\times X\to \m A$ is called a {\rm $C^*$-algebra valued metric} on $X$ if the following conditions hold: 
		\begin{itemize}
			\item [(a)] $d(x,y)\geq 0$ for each $x,y\in X.$ In addition, $d(x,y)=0~\Leftrightarrow~x=y;$  
			\item [(b)] $d(x,y)=d(y,x)$ for each $x,y\in X;$
			\item [(c)] $d(x,y)\leq d(x,z)+d(z,y)$ for each $x,y,z\in X.$
		\end{itemize}
		In this case, we say that $(X,\m A,d)$ is a {\rm $C^*$-algebra valued metric space}.
	\end{definition}
	
	\begin{definition}
		Assume $(X,\m A,d)$ to be a $C^*$-algebra valued metric space. A function $T:X\to X$ is called {\rm $C^*$-algebra valued contraction} if
		there exists an $A\in\m A$ such that $\|A\|<1$ and 
		$$d(Tx,Ty)\leq A^*d(x,y)A~\mbox{for every}~x,y\in X.$$
	\end{definition}
	
	In \cite[Theorem 2]{ZS16}, Kadelburg and Radenovi proved that the $C^*$-algebra valued metric space version of Banach contraction principle \cite[Theorem 2.1]{ZLH14} is equivalent to the classical Banach contraction principle.
	Using the similar techniques as in \cite{R07}, in the next section we give an easy an direct way to prove the $C^*$-algebra valued metric space version of Banach contraction principle.
	
	\section{Banach contraction principle}
	
	Let $(X,\m A,d)$ be a $C^*$-algebra valued metric space and let $T:X\to X$ be a $C^*$-algebra valued contraction. Then for each $x_1,x_2\in X$ from the triangle inequality, we have
	\begin{align*}
	d(x_1,x_2)&\leq d(x_1,T(x_1))+d(T(x_1),T(x_2))+d(T(x_2),x_2)
	\\& \leq d(x_1,T(x_1))+A^* d(x_1,x_2)A+d(T(x_2),x_2).
	\end{align*}
	Taking the $C^*$-norm on both the sides and using $\|A\|<1,$ we get the following inequality:\\
	{\bf Fundamental contraction inequality:}\\
	If $(X,\m A,d)$ is a $C^*$-algebra valued metric space and $T:X\to X$ is a $C^*$-algebra valued contraction, then
	\begin{equation}\label{eq}\|d(x_1,x_2)\|\leq \frac{1} {1-\|A\|^2}\|d(x_1,T(x_1))+d(T(x_2),x_2)\|. \end{equation}
	Therefore $T$ cannot have two distinct fixed points $x_1$ and $x_2.$
	\begin{corollary}
		The contraction $T$ cannot have more than one fixed point. 
	\end{corollary}
	
	\begin{proposition}
		Let $(X,\m A,d)$ be a $C^*$-algebra valued metric space and let $T:X\to X$ be a $C^*$-algebra valued contraction. Then for each $x\in X,$ the sequence $\{T^n(x)\}$ is a Cauchy sequence.
	\end{proposition}
	\begin{proof}
		Substitute $x_1=T^n(x)$ and $x_2=T^m(x)$ in the inequality \ref{eq}, we obtain
		\begin{align}\label{eq2}
		\|d(T^n(x),T^m(x))\|&\nonumber\leq \frac{1} {1-\|A\|^2}\|d(T^n(x),T^{n+1}(x))+d(T^{m+1}(x),T^m(x))\|\\
		&\nonumber \leq \frac{1} {1-\|A\|^2}\|A^{*n}d(x,T(x))A^n+A^{*m} d(T(x),x))A^m\|\\
		& \leq \frac{\|A\|^{2n}+\|A\|^{2m}} {1-\|A\|^2} \|d(x,T(x))\|.
		\end{align}
		Therefore using $\|A\|<1,$ the preceding inequality implies that $\{T^n(x)\}$ is a Cauchy sequence.
	\end{proof}
	
	When $(X,\m A,d)$ is complete, the Cauchy sequence $\{T^n(x)\}$ converges to a point $p\in X.$ It is trivial that $p$ is a fixed point of $T.$
	Taking $m\to \infty,$ in inequality \ref{eq2}, we have the following version of Banach fixed point theorem (see \cite[Theorem 2.1]{ZLH14}):
	
	\begin{theorem}
		Suppose $(X,\m A,d)$ is a $C^*$-algebra valued complete metric space and $T:X\to X$ is a $C^*$-algebra valued contraction. Then $T$ has a unique fixed point $p\in X.$ 
	\end{theorem}

\end{document}